\documentclass[reqno,a4paper,12pt]{amsart}
\usepackage{amssymb,amsmath,amscd,amstext,amsthm,amsfonts}
\usepackage[]{graphicx}
\usepackage{epstopdf}
\usepackage{setspace} 
\usepackage{subfig}
\usepackage[mathscr]{eucal}
\usepackage[ansinew]{inputenc} 

\numberwithin{equation}{section}

\newtheorem{theorem}{Theorem}[section]
\newtheorem{proposition}[theorem]{Proposition}

\newtheorem{corollary}[theorem]{Corollary}
\newtheorem{lemma}[theorem]{Lemma}

{\theoremstyle{definition}
{
\newtheorem{remark}[theorem]{Remark}

\newtheorem{defn}[theorem]{Definition}
}}

\newcommand{\cal}{\mathcal}

\newcommand{\PP}{{\cal P}}

\newcommand{\Rr}{{\mathbb{R}}}

\newcommand{\comment}[1]{}

\begin{document}
\title[Ergodicity of slap maps]{Ergodicity of polygonal slap maps}
\date{\today}

\author[Del Magno]{Gianluigi Del Magno}
\address{CEMAPRE, ISEG\\
Universidade de Lisboa\\
Rua do Quelhas 6, 1200-781 Lisboa, Portugal}
\email{delmagno@iseg.utl.pt}

\author[Lopes Dias]{Jo\~ao Lopes Dias}
\address{Departamento de Matem\'atica and CEMAPRE, ISEG\\
Universidade de Lisboa\\
Rua do Quelhas 6, 1200-781 Lisboa, Portugal}
\email{jldias@iseg.utl.pt}

\author[Duarte]{Pedro Duarte}
\address{Departamento de Matem\'atica and CMAF \\
Faculdade de Ci\^encias\\
Universidade de Lisboa\\
Campo Grande, Edificio C6, Piso 2\\
1749-016 Lisboa, Portugal 
}
\email{pduarte@ptmat.fc.ul.pt}

\author[Gaiv\~ao]{Jos\'e Pedro Gaiv\~ao}
\address{CEMAPRE, ISEG\\
Universidade de Lisboa\\
Rua do Quelhas 6, 1200-781 Lisboa, Portugal}
\email{jpgaivao@iseg.utl.pt}

\begin{abstract}
Polygonal slap maps are piecewise affine expanding maps of the interval obtained by projecting the sides of a polygon along their normals onto the perimeter of the polygon. These maps arise in the study of polygonal billiards with non-specular reflections laws. We study the absolutely continuous invariant probabilities of the slap maps for several polygons, including regular polygons and triangles. We also present a general method for constructing polygons with slap maps having more than one ergodic absolutely continuous invariant probability.
\end{abstract}

\maketitle



\section{Introduction}
\label{sec:introduction}

Piecewise expanding maps of the interval are one-dimensional dynamical systems with a rich dynamics. The ergodic properties of these maps are well understood. Under proper conditions, they admit finitely many ergodic absolutely continuous invariant probabilities (acip's), and each ergodic component decomposes into a finite number of mixing components cyclically permuted by the map~\cite{BG97}. In this paper we are interested in a class of piecewise expanding maps which we call polygonal slap maps, appearing in the study of polygonal billiards. We focus on the analysis of the acip's of these maps for triangles and regular polygons, determining in particular the exact number of their ergodic and mixing components.

Denote by $\PP_d$ the set of $d$-gons that are not self-intersecting. Consider a polygon $P\in\PP_d$ with perimeter $L$, and let $s\in[0,L]$ be the arc length parameter of $\partial P$. Let $ \rho $ be the ray directed inside $ P $ and orthogonal to $\partial P$ at $s$, and let $ s' $ be the closest point to $ s $ among the points in $ \rho \cap P $ (see Fig.~\ref{fig:1a}). The map $\psi_P $ defined by $ s \mapsto s'$ is a piecewise affine map, called the {\em slap map} of $P$. If $P$ does not have parallel sides, then $\psi_P$ is a piecewise affine expanding map of the interval (see Fig.~\ref{fig:1b}).

\begin{figure}[t]
\subfloat[]{\label{fig:1a}\includegraphics[scale=.7]{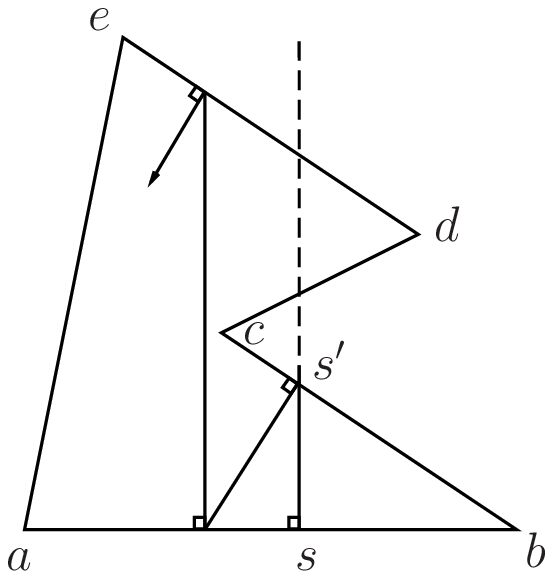}}
\label{prova}
\hspace{1cm}
\subfloat[]{\label{fig:1b}\includegraphics[scale=.65]{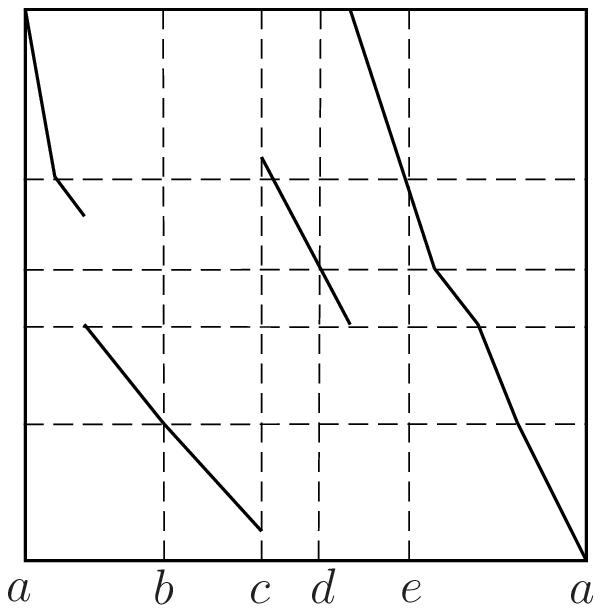}}
\caption{(A) Slap map $\psi_P$. (B) Graph of $\psi_P$.}
\label{fig:slapmap}
\end{figure}

Slap maps have been introduced in~\cite{markarian10} as a tool to study billiards with a strongly contracting reflection law (see also~\cite{arroyo09}). In fact, slap maps can be thought of as a billiard for which the reflection angle measured with respect to the normal is always zero. In~\cite{MDDGP13}, the study of slap maps was key in establishing the existence of hyperbolic attractors with finitely many Sinai-Ruelle-Bowen measures for polygonal billiards with strongly contracting reflection laws. The aim of this paper is to study the ergodic properties of polygonal slap maps, which will be used in future work to derive ergodic properties of polygonal billiards~\cite{MDDGP13-2}. 

The following are the main results of this paper. The first one deals with regular polygons and answers negatively a question formulated by Markarian, Pujals and Sambarino in~\cite[Section 5.1]{markarian10}.

\begin{theorem} 
\label{main thm:regular}
Let $ P $ be a regular polygon with an odd number $ d $ of sides. Then $ \psi_{P} $ has a unique ergodic acip if $ d=3 $ or $ d=5 $, and exactly $ d $ ergodic acip's if $ d \ge 7 $. Every acip's has $2^{m(d)}$ mixing components for every odd $d\geq 3$, where $m(d)$ is the integer part of $ -\log_2(-\log_2 \cos(\pi/ d)) $. In particular, $m(3)=0$, $m(5)=1$ and $m(7)=2$.
\end{theorem}

Our second main result concerns slap maps of general triangles.

\begin{theorem}
\label{main thm:triangles}
The slap map $\psi_P$ of a triangle $P$ has a unique ergodic acip. If $ P $ is acute, then the acip is mixing and is supported on the whole interval. Otherwise, the acip has an even number of mixing components.
\end{theorem}

From the previous theorems, it emerges that the uniqueness of the acip is not a typical property among polygonal slap maps. To further explore this fact, we present a procedure for constructing polygons with a number of sides greater than three whose slap map has several ergodic acip's. In particular, we show that for any $n\geq 2$ there exists a convex $3n$-gon whose slap map has $n$ ergodic acip's with even mixing period. Moreover, we study a bifurcation of certain periodic orbits yielding a mechanism for the creation of ergodic components. We apply it to show that there exist quadrilaterals with a mirror symmetry that have at least two acip's. 


The paper is organized as follows. In section~\ref{sec:expanding maps} we review several results on piecewise expanding maps of the interval concerning the existence of acip's and their spectral decomposition. Since slap maps of regular polygons are extensions of certain Lorenz maps, in section~\ref{sec:expanding maps}, we also briefly review the main results concerning the renormalization of Lorenz maps. The proof of Theorems~\ref{main thm:regular} and ~\ref{main thm:triangles} is contained in sections~\ref{sec:regular polygons} and ~\ref{sec:triangles}, respectively. Finally, in section~\ref{sec:non ergodic}, we describe in detail a method to construct polygons with more than one ergodic acip's.

\section{Piecewise Expanding Interval Maps} 
\label{sec:expanding maps}


Let $I$ be a closed interval. A map  $f \colon I\to I$ is called {\it piecewise expanding} if there exist a constant $\sigma>1$ and closed intervals $I_0,\ldots, I_m$  such that  
\begin{enumerate}
\item $I=\bigcup_{i=0}^m I_i$ and ${\rm int}(I_i)\cap {\rm int}(I_j)=\emptyset$ for $i\neq j$,
\item $f\vert_{{\rm int}(I_i)}$ is $ C^{1} $, and $ | f'|_{{\rm int}(I_i)}| \ge \sigma$ for each $ i $,
\item $ 1/|f'| $ is a function of bounded variation on each $ {\rm int}(I_i) $.
\end{enumerate}

Throughout the paper, we will use the standard abbreviation \emph{acip} for an invariant probability measure of $f$ that is absolutely continuous with respect to the Lebesgue measure of $ I $. 

\subsection{Spectral decomposition}

In the following theorem, we summarize well known results on ergodic properties of piecewise expanding maps (for example, see~\cite[Theorems 7.2.1, 8.2.1 and 8.2.2]{BG97}).

\begin{theorem}\label{th:acip} \label{th:folk}
Let $f$ be a piecewise expanding map. Then 
\begin{enumerate}
\item $ f $ has exactly $ j $ ergodic acip's $ \nu_{1},\ldots,\nu_{j} $ for some $ 1 \le j \le m $,
\item the density of each $ \nu_{i} $ has bounded variation, and its support $ \Lambda_{i} $ consists of finitely many intervals,
\item for each $ i $, there exist an integer $ k_{i} \ge 1 $ and disjoint measurable sets such that $\Lambda_i=\Sigma_{i,1} \cup \cdots \cup \Sigma_{i,k_i}$, the sets $ \Sigma_{i,1},\ldots,\Sigma_{i,k_i} $ are cyclically permuted by $ f $, 
and $ (f^{k_i}|_{\Sigma_{i,1}},\nu_{i}|_{\Sigma_{i,1}}) $ is exact.
\end{enumerate}
\end{theorem}


%

The sets $ \Sigma_{i,j} $ and the constant $ k_{i} $ are called the \emph{mixing components} of $ f $, and the \emph{mixing period} of the cycle $ \Sigma_{i,1},\ldots,\Sigma_{i,k_i} $, respectively. 


\subsection{Lorenz maps}

\begin{defn}
A map $f \colon[0,1]\to[0,1]$ is called a {\it Lorenz map} if $f$ has a unique discontinuity point $0<c<1$ such that $f(c^-)=1$, $f(c^+)=0$, and $f$ is monotonic increasing in both intervals $[0,c)$ and $(c,1]$.
\end{defn}


A Lorenz map $f \colon[0,1]\to[0,1]$ is {\it renormalizable} if there exist $\ell, r> 1$ and a proper subinterval $[a,b]\subset [0,1]$ such that
the transformation $g \colon[a,b]\to [a,b]$ defined by
$$ g(x)=\left\{\begin{array}{lcr}
f^\ell(x) & \text{ if } & x\in [a,c) \\
f^r(x) & \text{ if } & x\in (c,b] \\
\end{array} \right. $$
is also a Lorenz map. We call $[a,b]$ a \textit{renormalization interval}. If $f$ is not renormalizable it is said to be {\it prime}. 

A renormalization $g=(f^\ell,f^r)$ of $f$ is {\it minimal} if for any other renormalization $g'=(f^{\ell'},f^{r'})$ of $f$ we have $\ell'\geq \ell$ and $r'\geq r$. For every renormalizable Lorenz map $f$ we define $\mathscr{R}f$ to be its minimal renormalization. Finally, we say that a Lorenz map $f$ is {\it $m$-times renormalizable} with $0\leq m \leq \infty$ if $\mathscr{R}^kf$ is renormalizable for $0\leq k< m$ and $\mathscr{R}^mf$ is prime.


Let $f \colon[0,1]\to[0,1]$  be a piecewise expanding Lorenz map.

\begin{theorem}[Glendinning and Sparrow~\cite{GS93}]\label{th:GS}
If $f$ is prime then it admits a unique mixing acip with support equal to the interval $[0,1]$.
\end{theorem}

Consider now the family of piecewise affine Lorenz maps $f_a\colon[0,1]\to[0,1]$ with $1<a\leq2$ given by
$$
f_a(x)=a(x-1/2)\pmod{1}.
$$
We call $f_a$ a {\it centrally symmetric} map.
Let $m=m(a)$ be the unique nonnegative integer such that 
$$
2^{2^{-m-1}}<a\leq2^{2^{-m}}.
$$
A centrally symmetric interval is a subinterval of $[0,1]$ centered at $1/2$. From the results of W. Parry~\cite{parry79} we get

\begin{theorem}
\label{th:lorenz}
 $f_a$ is $m$-times renormalizable with $\mathscr{R}^k f_a=f_a^{2^k}|_{J_k}$ for $k=0,1,\ldots,m$, where the renormalization intervals $J_k$ form a nested sequence of centrally symmetric intervals,
$$
1/2\in J_m\subset J_{m-1}\subset\cdots\subset J_1\subset J_0=[0,1].
$$
\end{theorem}

As an immediate consequence of Theorems~\ref{th:GS} and \ref{th:lorenz} we have the following result.

\begin{corollary}\label{co:lorenz}
$f_a$ has a unique ergodic acip and $2^m$ mixing components.
\end{corollary}


\begin{remark}
When $\sqrt{2}<a\leq2$, we have $m(a)=0$. Thus $f_a$ is mixing.
\end{remark}

\subsection{Polygonal slap maps}

%

\begin{defn}
We say that a polygon $ P \in\PP_{d} $ has \emph{parallel sides facing each other} if $ P $ contains a segment intersecting orthogonally two of its sides and having no other intersections with $ \partial P $.
\end{defn}

If a polygon $ P \in \PP_{d} $ does not have parallel sides facing each other, then the polygonal slap map $ \psi_{P} $ introduced in Section~\ref{sec:introduction} is piecewise expanding (see~Fig.~\ref{fig:slapmap}). On the other hand, if $ P $ has parallel sides facing each other, then $ \psi_{P} $ is not expanding, because the restriction of $ \psi^{2}_{P} $ to a subinterval of the domain of $ \psi_{P} $ is equal to the identity (i.e. $ \psi_{P} $ has a continuous family of periodic points of period two). 


\begin{corollary}\label{co:acipslap}
Suppose that $ P \in \PP_{d} $ does not have parallel sides facing each other. Then the map $ \psi_{P} $ has exactly $ j $ ergodic acip's for some $ 1 \le j \le d $. Moreover, every acip of $ f $ is supported on finitely many intervals. 
\end{corollary}

\begin{proof}
Since $ \psi_{P} $ is piecewise expanding, the claim follows from Theorem~\ref{th:folk}.
\end{proof}

\begin{remark}
The above maximum number of acip's cannot be improved. Indeed, in Theorem~\ref{th:Pd7} we show that if $ P \in \PP_{d} $ is regular and $ d \ge 7 $ is odd, then $ \psi_{P} $ has exactly $ d $ ergodic acip's.
\end{remark}




\section{Regular Polygons}
\label{sec:regular polygons}

In this section we characterize the ergodicity of slap maps on regular polygons with an odd number of sides.

Denote by $P_d$ the regular $d$-gon, $d\geq 3$. Suppose that the sides of $P_d$ have unit length. The slap map $\psi_{P_d}$ is a one dimensional piecewise affine map of the interval $[0,d]$ with $d$ points of discontinuity. The symmetries of the regular $d$-gon imply that 
$$
\psi_{P_d}(x+1)-1=\psi_{P_d}(x)\pmod{d}\,.
$$


Passing to the quotient we obtain the {\it reduced slap map} $\phi_{P_d}:[0,1]\to[0,1]$ given by 
$$
\phi_{P_d}(x)=\begin{cases}
-\frac{1}{\beta_d}\,\left(x-\frac{1}{2}\right) \; ( {\rm mod}\, 1) \;&\text{if $d$ is odd }\\
1-x&\text{ otherwise }
\end{cases},
$$
where $\beta_d=\cos(\frac{\pi}{d})$ is an algebraic number. A detailed derivation of the reduced slap map can be found in \cite[Section 6]{MDDGP13}. Clearly, when $d$ is even, the slap map is an involution, i.e. every point is periodic of period two. On the other hand, when $d$ is odd, the slap map is expanding. Since it has $d$ points of discontinuity, by Corollary~\ref{co:acipslap}, $\psi_{P_d}$ has at most $d$ ergodic acips, each supported on a union of finitely many intervals.

For the remaining part of this section we shall assume that $d$ is an odd number and to simplify the notation we shall write $\phi_d$ and $\psi_d$ for the reduced slap map and slap map of $P_d$, respectively.

%

The following lemma explains the relation between the slap map and its reduced slap map. Let $\mathbb{Z}_d=\mathbb{Z}/d\mathbb{Z}$ be the group of integers modulo $d$.

\begin{lemma}\label{lem:slapconjugation}
The slap map $\psi_{P_d}$ is conjugated to the skew-product map $F_{d}:[0,1]\times\mathbb{Z}_d\to[0,1]\times\mathbb{Z}_d$ defined by 
$F_{d}(x,s)=\left(\phi_{d}(x),\sigma_x(s)\right)$
where
$$
\sigma_x(s)=s+\left[\frac{d}{2}\right]\delta(x)\quad\text{and}\quad\delta(x)=\begin{cases}-1&x<1/2\\1&x>1/2\end{cases}\,.
$$
\end{lemma}

\begin{proof}
Let $H:[0,1]\times\mathbb{Z}_d\to[0,d]$ be the map defined by $H(x,s)=x+s$. Then a simple computation shows that the following diagram commutes.
$$
\begin{CD}
[0,1]\times\mathbb{Z}_d @>F_{d}>> [0,1]\times\mathbb{Z}_d\\
@VVHV @VVHV\\
[0,d] @>\psi_{d}>> [0,d]
\end{CD}
$$
\end{proof}


\begin{remark}\label{rem:slapconjugation2}
The reduced slap map is a factor of the slap map, i.e. $\pi\circ F_{P_d}=\phi_{P_d}\circ \pi$ where $\pi:[0,1]\times\mathbb{Z}_d\to[0,1]$ is the projection $\pi(x,s)=x$.
\end{remark}

Let $\mu$ be an acip of the slap map $\psi_d$. We denote by $\hat{\mu}$ the pushforward measure by $H^{-1}$, i.e. $\hat{\mu}:=(H^{-1})_*\mu$, where $H$ is the conjugation in Lemma~\ref{lem:slapconjugation}. Clearly, $\hat{\mu}$ is an invariant measure for the skew-product $F_{d}$. It follows from Theorem~\ref{th:folk} and Remark~\ref{rem:slapconjugation2} that

\begin{lemma}\label{lem:muhat}
The reduced slap map $\phi_d$ has a unique ergodic acip $\nu_d$. Moreover, $\pi_*\hat{\mu}=\nu_d$ for any acip $\mu$ of the slap map $\psi_d$. 
\end{lemma}
%


 Let $\varphi_d=\varphi\circ\phi_d$ where $\varphi:[0,1]\to[0,1]$ is the involution $\varphi(x)=1-x$. Also, let $m=m(d)$ be the unique nonnegative integer such that $$2^{2^{-m-1}}<1/\beta_d\leq 2^{2^{-m}}\,.$$

\begin{lemma}\label{lem:varphin}

The following holds:
\begin{enumerate}
	\item $\varphi_d$ is a centrally symmetric piecewise affine Lorenz map,
\item $(\varphi_d)_*\nu_d=\nu_d$,
\item there exist centrally symmetric intervals 
$$
J_m\subset J_{m-1}\subset\cdots\subset J_1\subset J_0=[0,1]
$$
such that $\varphi_{d}^{2^k}|_{J_k}$ is a Lorenz map for $k=0,\ldots,m$ and $\varphi_{d}^{2^m}|_{J_m}$ is mixing.
\end{enumerate}
\end{lemma}

\begin{proof}
The first assertion of the lemma is a simple computation, so we omit it. To prove the second assertion, note that the reduced slap map is $\varphi$-symmetric, i.e. $\varphi\circ\phi_d=\phi_d\circ\varphi$, which implies that $\varphi_*\nu_d$ is also an acip of $\phi_d$. By uniqueness of $\nu_d$ we conclude that $\varphi_*\nu_d=\nu_d$. Thus, 
$(\varphi_d)_*\nu_d=(\varphi\circ\phi_d)_*\nu_d=\varphi_*\nu_d=\nu_d$.
Finally, the last assertion follows from Theorem~\ref{th:GS} and Theorem~\ref{th:lorenz}.
\end{proof}

\begin{remark}
By Corollary~\ref{co:lorenz}, the Lorenz map $\varphi_d$ has $2^m$ mixing components. Taking into account that $\phi_d=\varphi_d\circ\varphi$ and $\varphi^2=\mathrm{id}$, we conclude that the reduced slap map $\phi_d$ also has $2^m$ mixing components.
\end{remark}




For every integer $n\geq0$ we write $F_{d}^{n}(x,s)=\left(\phi_{d}^{n}(x),\sigma_x^{n}(s)\right)$ where $\sigma_x^{n}(s):=\sigma_{\phi_{d}^{n - 1}(x)}\circ \cdots \circ \sigma_x(s)$. Note that $\sigma_x^n$ is a translation on $\mathbb{Z}_d$, i.e.
$\sigma_x^n(s)=s+\alpha_n(x)$ where
$$
\alpha_n(x):=\left[\frac{d}{2}\right]\sum_{i=0}^{n -1}\delta(\phi_{d}^i(x))\,.
$$

\begin{lemma}\label{lem:dk}
For $k=0,\ldots,m$, the following holds:
\begin{enumerate}
	\item $F_d^{2^k}(\pi^{-1}(J_k))=\pi^{-1}(J_k)$,
	\item $\alpha_{2^k}(x)=a_k\delta(x)$ for every $x\in J_k$, where $a_0=\left[d/2\right]$, $a_1=-1$ and 
	$$a_k=0,\quad k=2,\ldots,m\,.$$

\end{enumerate}
\end{lemma}

\begin{proof}
Item (1) follows from (3) of Lemma~\ref{lem:varphin}. Now we prove (2). Let $J_k^-=J_k\cap[0,1/2)$ and $J_k^+=J_k\cap(1/2,1]$. Since $\phi_{d}^{2^k}$ is continuous on $J_k^{\pm}$, there exists $a_k^{\pm}\in\mathbb{Z}_d$ such that 
$\alpha_{2^k}(x)=a_k^{\pm}$ for every $x\in J_k^{\pm}$. A simple computation shows that $\alpha_{2^k}=-\alpha_{2^k}\circ\varphi$, which implies that $a_k^+=-a_k^-$. Thus, $\alpha_{2^k}(x)=a_k\delta(x)$ where $a_k:=a_k^+$. Clearly, $\alpha_1(x)=\left[d/2\right]\delta(x)$. So $a_0=\left[d/2\right]$. The remaining $a_k$ can be computed recursively from the equation $\alpha_{2^{k+1}}(x)=\alpha_{2^k}(x)+\alpha_{2^k}(\phi_{d}^{2^k}(x))$. This concludes the proof of the lemma.
\end{proof}

%
%

We now study each regular polygon separately. Given any interval $I\subset[0,1]$ let $I_s=I\times\{s\}$.

\subsection{Equilateral triangle}

Since $m(3)=0$, the reduced slap map $\phi_3$ of the equilateral triangle is mixing. It is easy to see that the same is true for the slap map $\psi_3$. 

\begin{theorem}\label{th:P3}
The slap map of the equilateral triangle has a unique mixing acip with support equal to the interval $[0,3]$.
\end{theorem}

\begin{proof}
Note that $\phi_3([0,1/2^{-}])=\phi_3([1/2^+,1])=[0,1]$. Since $F_3(x,s)=(\phi_3(x),s+\delta(x))$
we have that $F_3(I_s)=I_{s-1}\cup I_{s+1}$ for every $s\in\mathbb{Z}_3$.
This implies that $\psi_{3}$ is Markov with irreducible and aperiodic transition matrix. The claim follows.
\end{proof}

\subsection{Regular pentagon}

\begin{theorem}\label{th:P5}
The slap map of the regular pentagon has a unique ergodic acip with two mixing components.
\end{theorem}

First note that $m(5)=1$. By Lemma~\ref{lem:varphin}, $\varphi_5$ is $1$-time renormalizable. In fact, a simple computation shows that $\varphi_5$ is renormalizable on $J=[e,\varphi(e)]$ where $e=(3-\sqrt{5})/2$. 
Since $\varphi_5^2|_{J}$ is an expanding Lorenz map, there exists a unique $b\in [e,1/2)$ such that $\varphi_5^2(b)=1/2$. In fact, $b=(9-\sqrt{5})/16$. By $\varphi$-symmetry, $\varphi_5^2(\varphi(b))=1/2$. Let $B=[b,\varphi(b)]$. It is easy to see that $\varphi_5^2$ has a unique periodic orbit $\mathcal{O}_p=\{p,\varphi(p)\}$ of period two contained in $B$ (see Fig.~\ref{fig:slapmap_pentagon}). 
\begin{figure}[t]
\subfloat[]{\includegraphics[scale=.45]{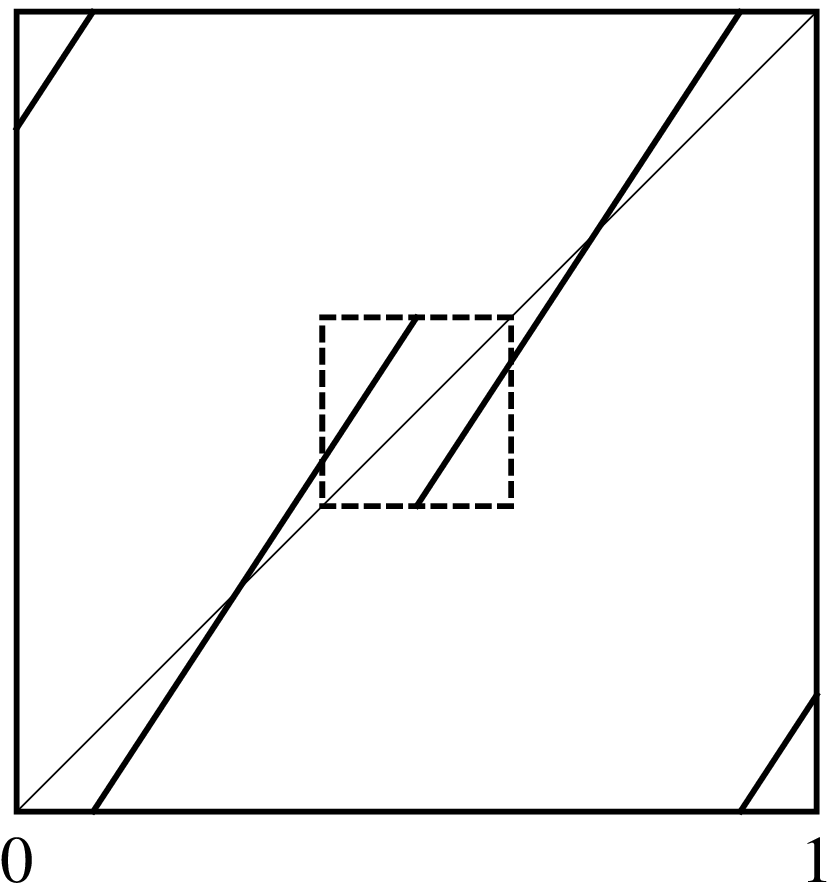}}
\hspace{1cm}
\subfloat[]{\includegraphics[scale=.45]{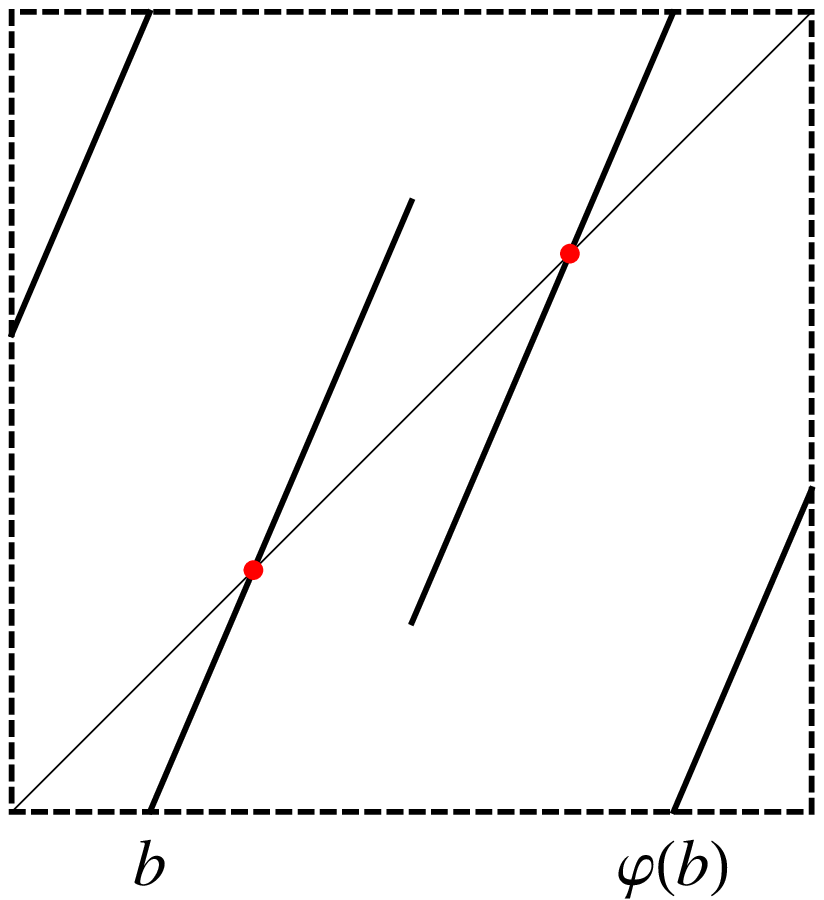}}
\caption{(A) Graph of $\varphi_5^2$. 
The centered dashed square is $J_1^2$. (B) Graph of $\varphi_5^4$ on the interval $J_1$. 
The points represent the periodic orbit $\mathcal{O}_p$ of $\varphi_5^2$.}
\label{fig:slapmap_pentagon}
\end{figure}

\begin{lemma}\label{lem:pentagon1}
For any interval $I\subset [0,1]$ such that $I\cap \mathcal{O}_p\neq\emptyset$, there exists an integer $n\geq1$ such that $J_s\subset F_5^{4n}(I_s)$ for every $s\in\mathbb{Z}_5$.
\end{lemma}

\begin{proof}

By Lemma~\ref{lem:dk}, $\alpha_2(x)=-\delta(x)$ for every $x\in J$. Using equation $\alpha_4(x)=\alpha_2(x)+\alpha_2(\phi_5^2(x))$ as in the proof of Lemma~\ref{lem:dk}, we get
$\alpha_4(x)=-2\delta(x)\chi_{J\setminus B}(x)$ for every $x\in J$.  
Thus, $\alpha_4(x)=0$ for every $x\in B$. This implies that \begin{equation}\label{eq:F54}
F_5^4(B_s)= J_s\quad \text{for every}\quad s\in\mathbb{Z}_5\,.
\end{equation}
Let $B^-=[b,1/2^-]$ and $B^+=\varphi(B^-)$. Note that $B=B^-\cup B^+$. Now let $I'=[p-\delta,p+\delta]$ and $\delta>0$ arbitrary small such that $I'\subset B^-$. Since $\phi_5^4$ is expanding and $p$ is a fixed point of $\phi_5^4$, $I'$ is a local unstable manifold of $p$. By iteration of the map $\phi_5^4$, the length of $\phi_5^{4n}(I')$ will increase until it is no longer strictly contained in $B^-$. Taking into account that $\alpha_4|_{B}=0$, we conclude that there exists an integer $j\geq1$ such that $B^-_s\subset F_5^{4j}(I'_s)$ for every $s\in\mathbb{Z}_5$. A simple computation shows that $\phi_5^4(1/2^{-})>\varphi(p)$. Thus $\varphi(p)\in\phi_5^{4}(B^-)$. So, we can repeat the same argument starting with any small interval containing $\varphi(p)$. Hence, there exists an integer $n\geq j$ such that 
$B_s\subset F_5^{4n}(I_s)$ for every $s\in\mathbb{Z}_5$ and
for every interval $I$ intersecting $\mathcal{O}_p$. By equation \eqref{eq:F54} we get the desired result.

\end{proof}

Let $\mu$ be an ergodic acip of the slap map $\psi_5$ and $\Lambda$ the support of $\mu$. By Theorem~\ref{th:folk},  we may decompose $\Lambda$ in mixing components,
$$
\Lambda=\Sigma_1\cup\cdots\cup\Sigma_k\,,
$$
such that $\psi_5^k|_{\Sigma_i}$ is mixing. Since $\phi_5^2|_{J}$ is mixing, the mixing period $k$ must be even. Let $H$ be the conjugation in Lemma~\ref{lem:slapconjugation}. 

\begin{lemma}\label{lem:pentagon2}
$H(\pi^{-1}(J))$ is a mixing component of $\mu$.
\end{lemma}

\begin{proof} By Lemma~\ref{lem:muhat}, there exists a mixing component $\Sigma$ of $\mu$ which satisfies $H^{-1}(\Sigma)\cap \mathcal{O}_p\neq\emptyset$. Note that $\Sigma$ is a union of finitely many intervals. This fact follows from (2) of Theorem~\ref{th:folk} applied to the map $\psi^k|_{\Lambda}$. Thus, we can decompose $H^{-1}(\Sigma)$ as follows 
$$
H^{-1}(\Sigma)=I_{1,s_1}\cup\cdots\cup I_{r,s_r}
$$ 
where $I_{j,s_j}=I_j\times\{s_j\}$, $s_j\in\mathbb{Z}_5$ and $I_j$ is a subinterval of $[0,1]$. By a previous observation we conclude that $I_j\cap \mathcal{O}_p\neq\emptyset$ for some $1\leq j\leq r$. Applying Lemma~\ref{lem:pentagon1} we get $J_{s_j}\subset H^{-1}(\Sigma)$. Now, by Lemma~\ref{lem:dk} we have  $\alpha_2(x)=-\delta(x)$ for every $x\in J$. Thus, 
$$
F_5^2(J_{s_{i,j}})\cap \left(\mathcal{O}_p\times\{s_{i,j}\pm1\}\right)\neq\emptyset\,.
$$
Applying again Lemma~\ref{lem:pentagon1}, we obtain $J_{s_{i,j}+i}\subset H^{-1}(\Sigma)$ for $i=-1,0,1$. Repeating the argument for $J_{s_{i,j}\pm 1}$ we conclude that $\pi^{-1}(J)\subset H^{-1}(\Sigma)$. By (1) of Lemma~\ref{lem:dk}, $F_5^k(\pi^{-1}(J))=\pi^{-1}(J)$. So we get equality and the proof is complete.

\end{proof}

\begin{proof}[Proof of Theorem~\ref{th:P5}]

By Lemma~\ref{lem:pentagon2}, any ergodic acip $\mu$ has a mixing component $\Sigma$ which is equal to $H(\pi^{-1}(J))$. Therefore, the slap map $\psi_5$ has at most one ergodic acip. Finally, since $F_5^2(\pi^{-1}(J_1))=\pi^{-1}(J_1)$, we conclude that $\psi_5$ has mixing period two.
%

\end{proof}

\subsection{Regular polygon $P_d$ with odd $d\geq7$}


\begin{figure}[t]
\subfloat[]{\includegraphics[scale=.55]{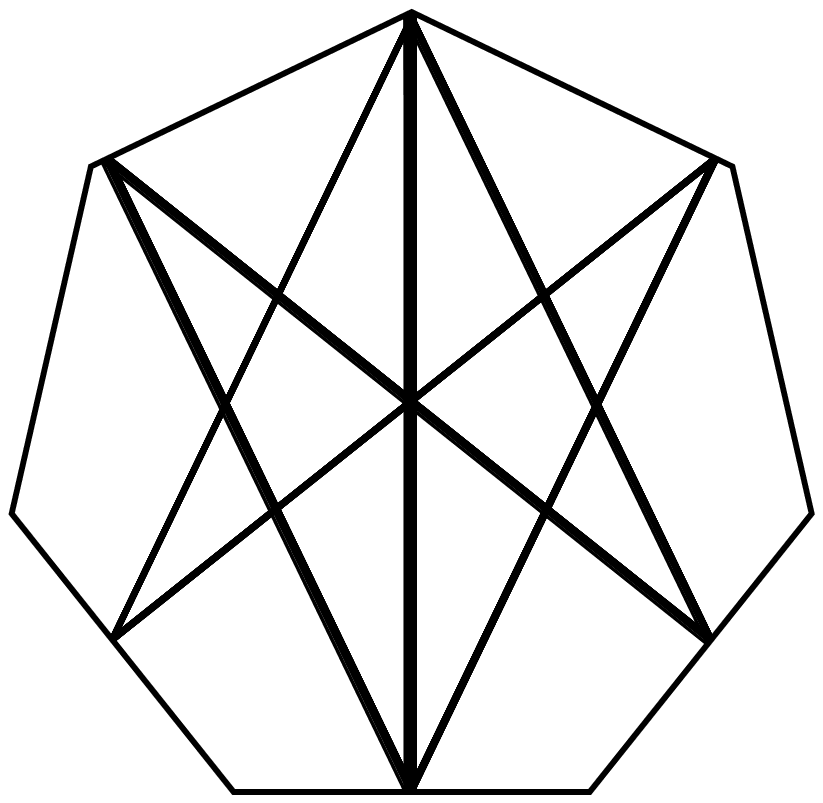}}
\hspace{1cm}
\subfloat[]{\includegraphics[scale=.5]{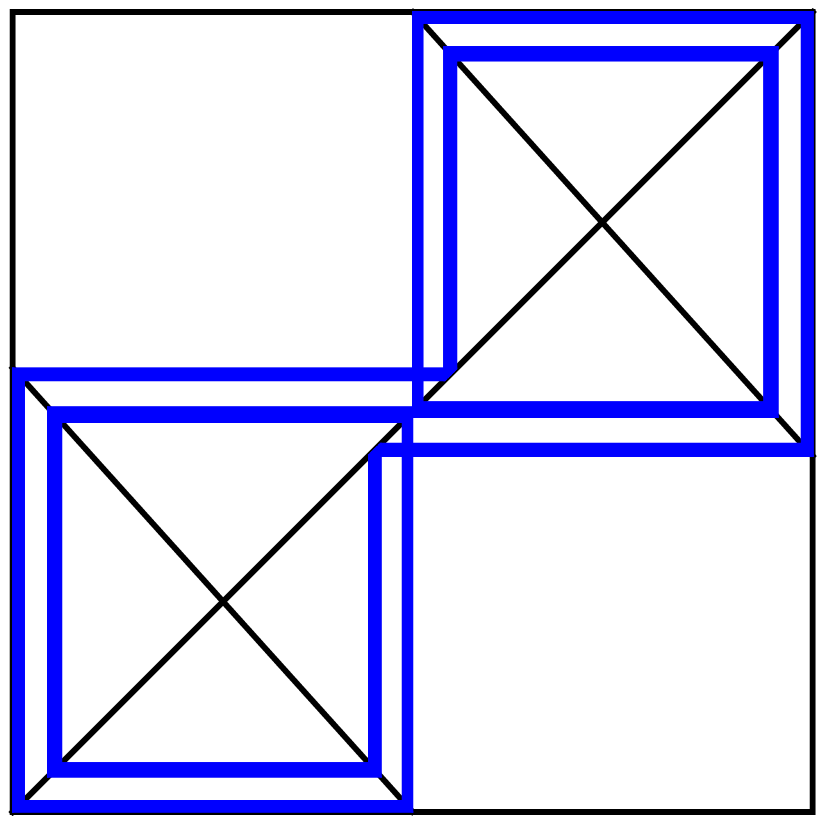}}
\caption{(A) One ergodic component of the slap map on the heptagon. (B) Graph of $\phi_7$ and an orbit on the support of $\nu_7$.}
\label{fig:7gon}
\end{figure}

\begin{theorem}\label{th:Pd7}
If $d\geq7$ and $d$ is odd then the slap map of $P_d$ has $d$ ergodic acips with $2^{m(d)}$ mixing components. 
\end{theorem}

\begin{proof}
When $d\geq7$ and $d$ is odd we have $m(d)\geq 2$. Since, by (2) of Lemma~\ref{lem:dk}, $\alpha_{2^m}(x)=0$ for every $x\in J_m$ we conclude that 
\begin{equation}\label{eq:F7}
F_d^{2^m}(x,s)= (\phi_d^{2^m}(x),s)\,,\quad\forall\,(x,s)\in J_m\times\mathbb{Z}_d\,.
\end{equation}
Therefore, the slap map has at least $d$ ergodic acips. However, the number of ergodic acips cannot be greater than $d$, the number of points of discontinuity of $\psi_d$. 
Thus, $\psi_d$ has exactly $d$ ergodic acips. Moreover, since $\phi_{d}$ has $2^m$ mixing components, by \eqref{eq:F7} the same is true for the slap map (see Fig.~\ref{fig:7gon}). 

\end{proof}

\section{Triangles}
\label{sec:triangles}

In this section we prove that the slap map of every triangle has a unique ergodic acip. 

Let $\Delta$ denote a triangle.

\begin{theorem}\label{th:triangle}
The slap map $\psi_\Delta$ has a unique ergodic acip. Moreover, if $\Delta$ is acute then $\psi_\Delta$ is mixing, and its acip is supported on the whole interval. Otherwise, $\psi_\Delta$ has a even number of mixing components. 
\end{theorem}
\begin{proof}We split the proof in two cases: (a) $ \Delta $ acute, and (a) $ \Delta $ not acute.

Case (a). Arguing as in the proof of Theorem~\ref{th:P3}, we see that the slap map of an acute triangle is a Markov map with irreducible and aperiodic transition matrix. Therefore, $\psi_\Delta$ has a unique mixing acip.

Case (b). Let $ I_{0},I_{1},I_{2} $ be the sides of $\Delta$. We assume that $ I_{0} $ is the longest side. The sets $ I_{0} $ and $ I_{1} \cup I_{2} $ are both forward invariant under the second iterate of the slap map. Moreover, $ \psi_\Delta^{2}|_{I_{0}} $ and $ \psi_\Delta^{2}|_{I_{1} \cup I_{2}} $ are both piecewise affine expanding maps with a single point of discontinuity. So, by Theorem~\ref{th:acip}, $ \psi_\Delta^{2}|_{I_{0}} $ and $ \psi_\Delta^{2}|_{I_{1} \cup I_{2}} $ have a unique ergodic acip, $\nu_0$ and $\nu_{12}$, respectively. Let $\mu$ be an acip of the slap map $\psi_\Delta$. Clearly $\mu_{I_0}=\nu_0$ and $\mu_{I_1\cup I_2}=\nu_{12}$, by uniqueness. To prove that $\mu$ is ergodic we argue as follows. Let $ J $ be an invariant subset of $ I_{0} \cup I_{1} \cup I_{2} $ for $ \psi_{\Delta}$. We can write $ J = A \cup B $ with $ A \subset I_{0} $ and $ B \subset I_{1} \cup I_{2} $. Then $ \psi_\Delta^{-2}(A) \cup \psi_\Delta^{-2}(B) = \psi_\Delta^{-2}(J) = J = A \cup B $. Since $ I_{0} $ and $ I_{1} \cup I_{2} $ are invariant under $ \psi_\Delta^{2} $, it follows that $ \psi_\Delta^{-2}(A) = A $ and $ \psi_\Delta^{-2}(B) = B $. But $ \psi_\Delta^{2}|_{I_{0}} $ and $ \psi_\Delta^{2}|_{I_{1} \cup I_{2}} $ are ergodic so that $\mu_{I_0}(A)=\mu_{I_1\cup I_2}(B)=1$. Thus $ \mu(J) = 1 $, which proves that $ \psi_{\Delta} $ is ergodic but not mixing with respect to $ \mu $. From previous observation, the number of mixing components has to be even. 
\end{proof}

\begin{remark}
When $\Delta$ is a right triangle, the slap map $\psi_\Delta$ is a Markov map with a transition matrix that is irreducible and periodic of period two. Thus, $\psi_\Delta$ has two mixing components. 
\end{remark}

\section{More examples of non-ergodic slap maps}
\label{sec:non ergodic}
In the previous section we proved that every regular polygon with an odd number $ d \ge 7 $ of sides has exactly $ d $ ergodic acip's. In this section we provide more examples of polygons whose slap map has more than one ergodic acip's. For brevity, we will refer to these polygons as `non-ergodic polygons'. 

It is quite easy to construct examples of non-ergodic polygons with more than five sides. For example, by intersecting two obtuse triangles, we obtain an hexagon with two ergodic acip's of even mixing period. This can be inferred immediately from Fig.~\ref{non:ergodic:hexagon}. The same construction permits to obtain a more general result.
\begin{figure}[h]
\begin{center}
\includegraphics*[trim=20mm 50mm 20mm 50mm, clip, width=5.5cm]{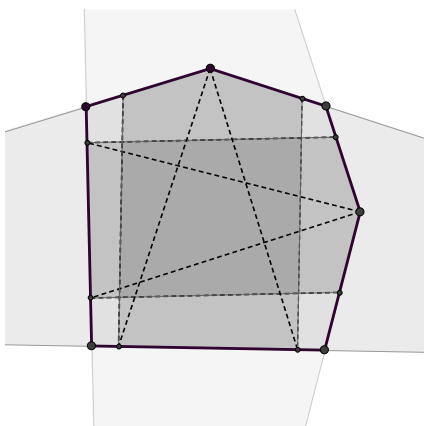} 
\end{center}
\caption{A non-ergodic hexagon.}
\label{non:ergodic:hexagon}
\end{figure} 
 

\begin{proposition}
For every integer $n\geq2$ there exists a convex $3n$-gon whose slap map has $n$ ergodic components with even mixing period. 
\end{proposition}

It is much harder to construct examples of non-ergodic quadrilaterals and pentagons.
In the rest of this section, we explain how to construct non-ergodic quadrilaterals. The basic idea behind our construction is to generate a kind of bifurcation for periodic orbits of the slap map implying the implosion, and hence the localization, of an ergodic component.

\begin{defn}
The orbit of a vertex $c$ of $P \in \PP_{d} $ is called a {\em doubling orbit} of type $(k,m)$ with $ k \ge 1 $ and $ m \ge 1 $ if 
	\begin{enumerate}
	\item $\psi_P^k(c^-)=\psi_P^k(c^+)=p$,
	\item $\psi_P^m(p)=c$,
	\item $\psi_P^i(c^\pm)$ is not a vertex for $i=1,\ldots, k+m-1 $,
	\end{enumerate}
where $\psi_P^k(c^\pm) := \lim_{x\to c^\pm} \psi_P^k(x)$.
\end{defn}

The {\em itinerary} of this doubling orbit is a triple $(\gamma_+,\gamma_-,\eta)$
where $\gamma_\pm$ are itineraries of $\{c^\pm, \psi_P(c^ \pm),\ldots, \psi_P^k(c^\pm)=p\}$,
and $\eta$ is the itinerary of $\{ p, \psi_P(p),\ldots, \psi_P^ m(p)\}$.

\begin{defn}
	Given a finite itinerary  $\gamma$, let $\psi_P^\gamma$ denote the composition of the branches of $\psi_P$ along $\gamma$.
\end{defn}

\begin{defn}\label{Pi:def}
Consider a doubling orbit $\mathscr{O}$ starting at a vertex $c$ with itinerary $(\gamma_+,\gamma_-,\eta)$. Let us denote by $c_{P'}$ the continuation of the vertex $c$ for a nearby polygon $P'$. We say that $\mathscr{O}$ is {\em generic} if and only if the map $\Pi \colon\mathcal{U} \to \Rr^2 $ defined in a neighborhood $\mathcal{U}$ of $P$ by $$\Pi(P')=\left(\,\psi_{P'}^{\gamma_+}(c_{P'})-(\psi_{P'}^{\eta})^{-1}(c_{P'}), \psi_{P'}^{\gamma_-}(c_{P'})-(\psi_{P'}^{\eta})^{-1}(c_{P'})\, \right)$$ has derivative $D\Pi_P$ of rank $2$ (i.e., maximal rank).	
\end{defn}

\begin{proposition}
Given polygon $P\in\mathcal{P}_d$ with a generic doubling orbit $\mathscr{O}$ of type $(k,m)$, there exists a codimension $2$ submanifold $\Sigma_{P} \subset \mathcal{P}_d$ passing through $P$ consisting of polygons with a doubling orbit of type $(k,m)$ that is a continuation of $\mathscr{O}$.
\end{proposition}

\begin{proof}
Define $ \Sigma_{P}=\{\, P'\in \mathcal{U}\,:\,  \Pi(P')=(0,0)\,\}$.
$\Sigma_{P}$ is a  codimension $2$ submanifold because
$\mathscr{O}$ is generic.
\end{proof}

\begin{proposition}\label{generic:fork:bif:thm}
Given a generic doubling orbit $\mathscr{O}$ of type $(k,m)$ for a polygon $P\in\PP_d$, and a surface $S\subset \PP_d$ transversal to $\Sigma_{P} $ at $P$, there is a neighborhood $\,\mathcal{U}$ of $P$ and cone $\Gamma\subset S$ with apex $P$ such that for every $P'\in \Gamma\cap \,\mathcal{U}$, the ergodic component of $\psi_{P'}$ containing the vertex in $\mathscr{O}$ is a periodic attractor $\Lambda$ of period $k+m$. If $k+m$ is even, then $\Lambda$ is  bounded between two periodic orbits of the same period. If $k+m$ is odd, then $\Lambda$ is bounded by a single periodic orbit of  period $2(k+m)$.
\end{proposition}

\begin{proof} Denote by $\Pi_+$ and $\Pi_-$
the  components of the map $\Pi:\mathscr{U}\to\Rr^2$
introduced in Definition~ \ref{Pi:def}.
We study separately the two cases: 1) $n+k$ even and 2) $n+k$ odd.

Case 1. Define the cone
$$\Gamma=\{\, P'\in S\,:\, \Pi_+(P')<0,\; \Pi_+(P')>0\,\}\;.$$
This cone is non-empty, because $\mathscr{O}$ is generic and the surface $S$
is transversal to $\Sigma$ at $P$.
Given $P'\in\Gamma$ consider the iterate of the slap map $T_{P'}:=\psi_{P'}^{m+k}$ restricted to a neighborhood of
the point $x(P'):=(\psi_{P'}^\eta)^{-1}(c_{P'})$.
The map $T_{P'}$ is piecewise increasing because $m+k$ is even.
At the bifurcation, $T_P$ is continuous at $x(P)$, while for
$P'\in\Gamma$ one has first kind discontinuities at $ x=x(P') $ (see Fig.~\ref{fork:bifurcation})
$$ \lim_{t\to x^+} T_{P'}(t) < T_{P'}(x) < \lim_{t\to x^-} T_{P'}(t).$$
\begin{figure}[t]
 \begin{center}
 \includegraphics*[scale=.3]{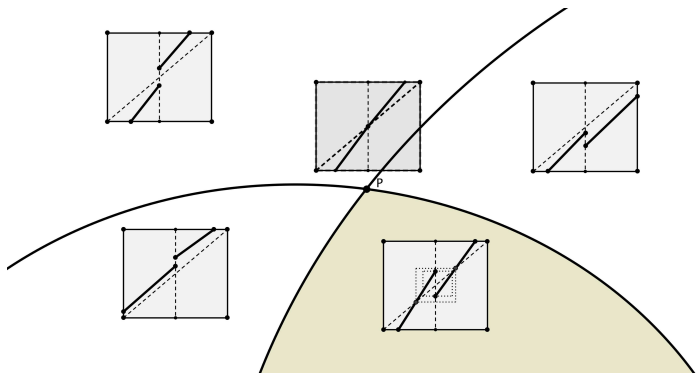} 
 \end{center}
 \caption{The parameter space around a doubling orbit bifurcation of even period.}
 \label{fork:bifurcation}
\end{figure}
Because the slopes of the branches of $T_{P'}$ are uniformly strictly greater than $1$,
if $P'$ is close to $P$ then the oscillation at the discontinuity point $x=x(P')$
is very small, and the map $T_{P'}$ has exactly two fixed points near $x$,
$t=x_-(P')$ and $t=x_+(P')$ such that $x_-(P')<x(P')<x_+(P')$.
Clearly  $[x_-(P'),x_+(P')]$ is invariant under $T_{P'}$,
and the restriction of $T_{P'}$ to this interval is a Lorenz map.
The fixed points $x_\pm(P')$ determine two periodic points of $\psi_{P'}$
with period $m+k$, and the invariant interval $[x_-(P'),x_+(P')]$ determines an ergodic component of the slap map  $\psi_{P'}$ which has mixing period $m+k$.

Case 2. Define the cone
$$\Gamma=\{\, P'\in S\,:\, \Pi_+(P')>0,\; \Pi_+(P')<0\,\}.$$
This cone is not empty for the same reason the cone $ \Gamma $ in Case 1 was not empty.
Given $P'\in\Gamma$ consider the iterate of the slap map $T_{P'}:=\psi_{P'}^{m+k}$ restricted to a neighborhood of
the point $x(P'):=(\psi_{P'}^\eta)^{-1}(c_{P'})$.
Now the map $T_{P'}$ is piecewise decreasing because $m+k$ is odd.
At the bifurcation, $T_P$ is continuous at $x(P)$, while $T_{P'}$  has a discontinuity of the first kind at $x=x(P')$,
$$ \lim_{t\to x^-} T_{P'}(t) < T_{P'}(x) < \lim_{t\to x^+} T_{P'}(t). $$
Because the slopes of the branches of $T_{P'}$ are uniformly strictly lesser than $-1$,
if $P'$ is close to $P$ then the oscillation at the discontinuity point $x=x(P')$
is very small, and the map $T_{P'}$ has exactly one periodic orbit of period $2$ near  $x$.
In other words there are
 $x_-=x_-(P')$ and $x_+=x_+(P')$ near $x$ such that 
$T_{P'}(x_-)= x_+$, $T_{P'}(x_+)=x_-$ and  $x_-<x<x_+$.
Clearly  $[x_-,x_+]$ is invariant under $T_{P'}$,
and the restriction of $T_{P'}^2$ to this interval is a Lorenz map.
The orbit $\{x_-,x_+\}$ determines a  periodic orbit of $\psi_{P'}$
with period $2(m+k)$, while the invariant interval $[x_-,x_+]$ determines an ergodic component of the slap map  $\psi_{P'}$ which has mixing period $2(m+k)$. 
\end{proof}

\subsection{Doubling orbits on kites}

In order to get two ergodic components we need two disjoint doubling orbits  simultaneously bifurcating. In this subsection, we explain how to obtain this phenomenon in kites, i.e. quadrilaterals with a mirror symmetry.


We consider a family of kites determined
by two angles $\alpha,\beta$ such that
$0<\beta<\frac{\pi}{4}<\alpha$ and $\alpha+\beta<\frac{\pi}{2}$.
They measure half of the quadrilateral angles bisected by the fixed unit diagonal.
We label the edges of a kite from $0$ to $3$
according to Fig.~\ref{fig:kites}(A).
Pairs of mirror symmetric edges have labels with the same parity.
\begin{figure}[h]
\subfloat[]{\includegraphics[scale=.55]{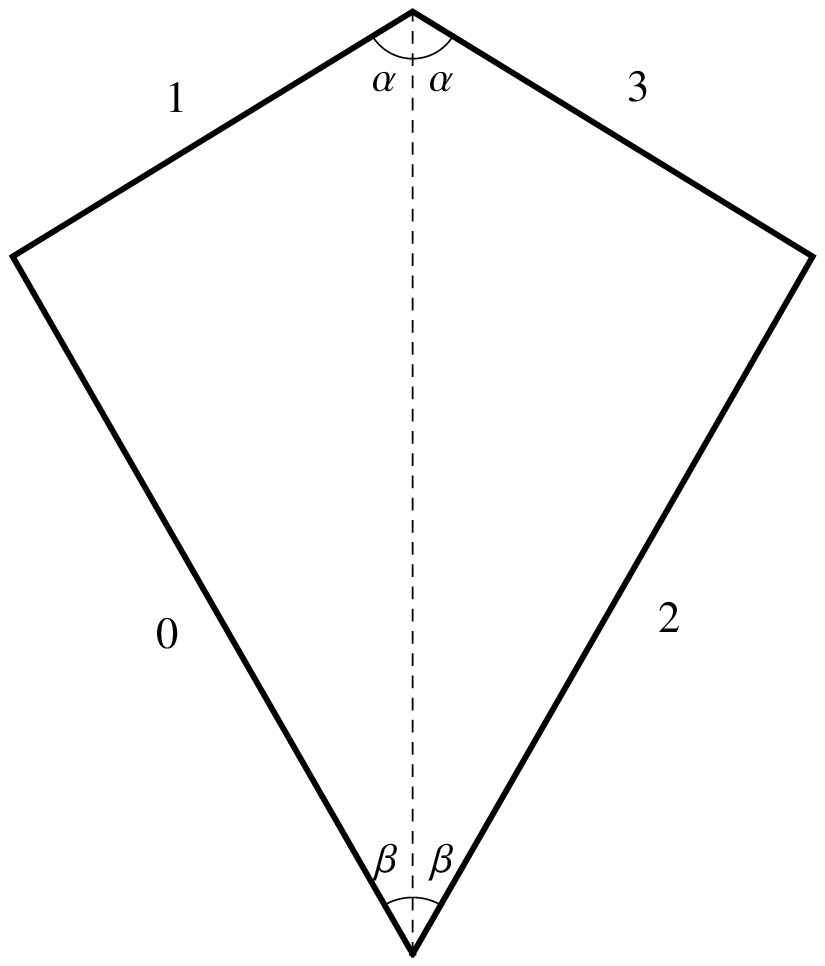}}
\hspace{1cm}
\subfloat[]{\includegraphics[scale=.55]{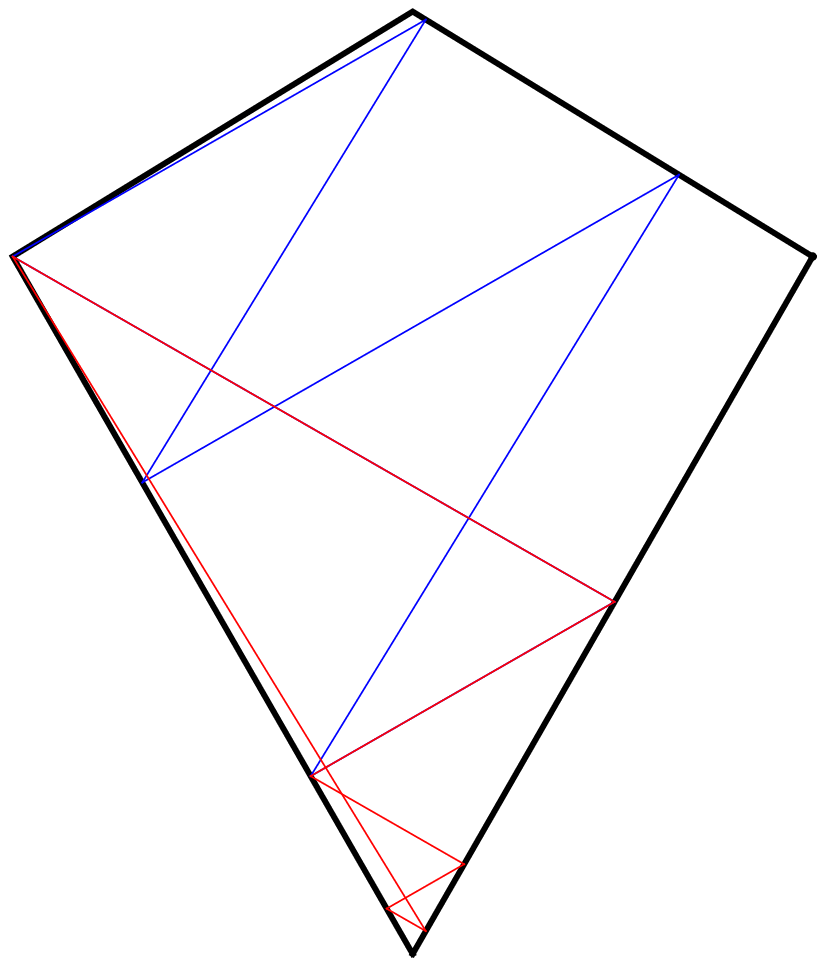}}
\caption{(A) Family of kites. (B) Asymmetric doubling orbit $\mathscr{O}$.}
\label{fig:kites}
\end{figure}

Let $K=K_{\alpha,\beta}$ be any member of this family.
The half-perimeter of $K$ is easily checked to be
$$\ell=\ell(\alpha,\beta):= \frac{\sin\alpha+\sin\beta}{\sin(\alpha+\beta)}\,.$$
By symmetry we can reduce
the slap map $\psi_K$ to the union of the edges labelled $0$ and $1$.
The reduced map
$\widetilde{\psi}_{K}:[0,\ell]\to [0,\ell]$ is the composition of the slap map $\psi_{K}$ with the reflexion around the given mirror.
The common vertex between the edges $0$ and $1$ has coordinate
$$c=c(\alpha,\beta):= \frac{\sin\alpha}{\sin(\alpha+\beta)}\,,$$ and
the orthogonal projection of the common vertex between the edges $2$ and $3$
onto the side $0$ has  coordinate
$$d=d(\alpha,\beta):= \frac{\cos(2\beta) \sin\alpha}{\sin(\alpha+\beta)}\,.$$

To simplify the notation we simply write $\psi$ for the reduced map $\widetilde{\psi}_{K}$. The reduced map $\psi$ has the following three branches:
\begin{enumerate}
\item[(i)] The branch that maps side $0$ to side $2$, which by reflexion is mapped to side $0$ again.
This branch is the one-to-one map 
$\psi_{00}:[0,d]\to [0,c]$ defined by
$$\psi_{00}(x)= \frac{x}{\cos(2\beta)}\;.$$

\item[(ii)] The branch that maps maps side $0$ to side $3$, which by reflexion is mapped to side $1$.
This branch is the one-to-one map 
$\psi_{01}:[d,c]\to [c,e]$ defined by
$$\psi_{01}(x)=  
\frac{x}{\cos(\alpha-\beta)} + \left( 1- \frac{\cos(2\beta)}{\cos(\alpha-\beta)}\right) \frac{\sin\alpha} {\sin(\alpha+\beta)}
\;,$$
where
$$e:=  \frac{ 1+\cos(\alpha-\beta)-\cos(2\beta) }{\cos(\alpha-\beta)}\frac{\sin\alpha} {\sin(\alpha+\beta)}\,.$$
 
\item[(iii)] Finally, the branch that maps side $1$ to side $2$, which by reflexion is mapped to side $0$.
This branch is the one-to-one map 
$\psi_{10}:[c,\ell]\to [p,q]$ defined by 
$$\psi_{10}(x)=  
\frac{x-c(\alpha,\beta)}{\cos(\alpha-\beta)} + \frac{\cos(\alpha+\beta)}{\cos(\alpha-\beta)}  \frac{\sin\alpha} {\sin(\alpha+\beta)}
\;,$$
where
$$
p:=  \frac{\cos(\alpha+\beta)\, \sin\alpha}{ \sin(\alpha+\beta)\,\cos(\alpha-\beta)}\quad\text{and}\quad
q:=  \frac{\cos \alpha }{\cos(\alpha-\beta)}\,.$$
\end{enumerate}

We say that an orbit of the slap map $\psi_K$ is \textit{asymmetric} if the corresponding trajectory on $K$ is asymmetric with respect to the vertical diagonal of $K$. We claim that there exists a kite $K$ such that the slap map $\psi_{K}$ has an asymmetric doubling orbit $\mathscr{O}$. We do not give an analytic proof of this claim. Instead, we provide strong numerical evidence of the existence of such orbits. We believe that using interval arithmetic our numerical computations can be transformed into a computer assisted proof.  

In the following, we are going to exhibit an asymmetric doubling orbit of type $(4,2)$ with itinerary 
$$
(\gamma_+,\gamma_-,\eta)=\left(\, (1,2,0,2,0),\, (0,3,0,3,0),\, (0,2,0) \,\right)\,.$$

Define  the $\Rr^2$-valued function 
$\Pi(\alpha,\beta):=(\Pi_+(\alpha,\beta),\Pi_-(\alpha,\beta))$, where the components $\Pi_\pm$ are given by
\begin{align*}
\Pi_+(\alpha,\beta)&:=  \psi_{00}^3\circ\psi_{10}(c)
- \psi_{00}^{-2}(c)\;,\\
  \Pi_-(\alpha,\beta)&:=   (\psi_{10}\circ \psi_{01})^2(c)
- \psi_{00}^{-2}(c)\;.
\end{align*}
We consider the  branch mappings $\psi_{00}$, $\psi_{01}$, $\psi_{10}$ and its inverses with the domains specified in items (i)-(iii) above.
The domain of the map $\Pi$, which we denote by $D_\Pi$, is the set of all $(\alpha,\beta)\in\Rr^2$
such that $0<\beta<\frac{\pi}{4}<\alpha$,\, $\alpha+\beta<\frac{\pi}{2}$, and
all compositions in the definition of $\Pi_\pm(\alpha,\beta)$ are meaningful, i.e.
\begin{align*}
&0<\psi_{00}^{-1}(c)<c\,,\quad 0<\psi_{00}^i\circ\psi_{10}(c)<d\,,\quad i=0,1,2\,,\\
&c<\psi_{01}(c)<\ell\,,\quad d<\psi_{10}\circ\psi_{01}(c)<c\,,\\
&c<\psi_{01}\circ\psi_{10}\circ\psi_{01}(c)<\ell\quad\text{and}\quad d<(\psi_{10}\circ\psi_{01})^2(c)<c\,.
\end{align*} 
Then any solution to the system of trigonometric polynomial equations
\begin{equation}\label{eq:P}
\Pi(\alpha,\beta)=(0,0)\,,\quad(\alpha,\beta)\in D_\Pi
\end{equation}
corresponds to a kite  whose slap map
possesses a doubling orbit $\mathscr{O}$  with the prescribed type and itinerary.

Using a computer algebra system to implement  Newton's Method one can check
that this problem has indeed a solution $(\alpha_0,\beta_0)$ whose
entries are approximately $$\alpha_0=1.021264\ldots\quad\text{and}\quad\beta_0=0.520719\ldots\,. $$
We can also verify that 
$\det D\Pi_{(\alpha_0,\beta_0)} = -24.321933 \ldots$ is non-zero. Hence, for the kite $K_0$ associated with the pair $(\alpha_0,\beta_0)$, the doubling orbit $\mathscr{O}$
is generic (see Fig.~\ref{fig:kites}(B)).
%
%
%
%
Denoting by $\mathscr{O}'$ the mirror image of $\mathscr{O}$, $\mathscr{O}'$ is another
generic doubling orbit with the given type and symmetric itinerary. 
It follows that the orbits $\mathscr{O}$ and $\mathscr{O}'$ are disjoint.
Thus, by Proposition~\ref{generic:fork:bif:thm}, there are arbitrarily small perturbations of $K_0$,
within the given family of kites, for which the slap maps have  ergodic
components  inside disjoint neighborhoods of the periodic orbits $\mathscr{O}$ and $\mathscr{O}'$. This shows that there are kites
whose slap map has at least two ergodic components. 
%

%
%
%


\section*{Acknowledgements}
The authors were supported by Funda\c c\~ao para a Ci\^encia e a Tecnologia through the Program POCI 2010 and the Project ``Randomness in Deterministic Dynamical Systems and Applications'' (PTDC-MAT-105448-2008).


\begin{thebibliography}{1}




\bibitem{arroyo09}
A. Arroyo, R. Markarian\ and\ D. P. Sanders, Bifurcations of periodic and chaotic attractors in pinball billiards with focusing boundaries, Nonlinearity {\bf 22} (2009), no.~7, 1499--1522.



\bibitem{BG97}
A. Boyarsky\ and\ P. G\'ora, {\it Laws of chaos}, Probability and its Applications, Birkh\"auser Boston, Boston, MA, 1997.





%
%



\bibitem{MDDGP12}
G.~Del Magno, J.~Lopes Dias, P.~Duarte, J.~P.~Gaiv\~ao\ and\ D.~Pinheiro, Chaos in the square billiard with a modified reflection law, Chaos {\bf 22}, (2012), 026106.

\bibitem{MDDGP13}
G.~Del Magno, J.~Lopes Dias, P.~Duarte, J.~P.~Gaiv\~ao\ and\ D.~Pinheiro, SRB measures for polygonal billiards with contracting reflection laws, Comm. Math. Phys., to appear.

\bibitem{MDDGP13-2}
G.~Del Magno, J.~Lopes Dias, P.~Duarte and J.~P.~Gaiv\~ao\, Ergodicity of polygonal billiards with contracting reflection laws, in preparation.



\bibitem{GS93} 
P. Glendinning\ and\ C. Sparrow, Prime and renormalisable kneading invariants and the dynamics of expanding Lorenz maps, Phys. D {\bf 62} (1993), no.~1-4, 22--50.

%
%
%

\bibitem{markarian10}
R. Markarian, E. J. Pujals\ and\ M. Sambarino, Pinball billiards with dominated splitting, Ergodic Theory Dynam. Systems {\bf 30} (2010), no.~6, 1757--1786.

\bibitem{parry79}
W. Parry, The Lorenz attractor and a related population model, in {\it Ergodic theory (Proc. Conf., Math. Forschungsinst., Oberwolfach, 1978)}, 169--187, Lecture Notes in Math., 729 Springer, Berlin. 



%










%
%


%
%
%
%
%
%
%
%
%
%
%
%

\end{thebibliography}
\end{document}